\documentclass[review]{elsarticle}
\usepackage{graphics,color}
\usepackage{amsmath}
\usepackage{amsfonts}
\usepackage{amssymb}
\usepackage{graphicx}
\input xy
\xyoption{all}
\usepackage{graphicx,amssymb,amsfonts,latexsym,amsmath,amsthm}

\usepackage{amsfonts}
\usepackage{amssymb}
\usepackage{eucal}
\usepackage{lineno,hyperref}
\modulolinenumbers[5]

\newtheorem{theorem}{Theorem}[section]
\newtheorem{lemma}[theorem]{Lemma}
\newtheorem{proposition}[theorem]{Proposition}
\newtheorem{definition}[theorem]{Definition}

\newtheorem{remark}[theorem]{Remark}

\newcommand{\cF}{{\cal F}}

\newcommand{\cP}{{\cal P}}

\def\cF{{\mathcal {F}}}

\def\cB{{\mathcal B}}

\def\cP{{\mathcal P}}

\def\R{{ \mathbb{R}}}

\everymath{\displaystyle}
\newcommand{\RE} {{\rm I \kern-2.8pt R} }

\begin{document}

\begin{frontmatter}

\title{ $\R \times BL^* $ Valued Consumer Resource Model}

\author{John Cleveland } 
\address{1200 US Highway 14 West, Richland Center WI 53581 }




\begin{abstract}
The ideas and techniques developed in \cite{CLEVACK, JC2} are applied to the basic pure selection (no mutation) parametric heterogeneous consumer resource model developed in \cite{SmithThieme} to derive a fully nonlinear resource dependent selection mutation $\R \times BL^*$ valued model. Where $BL^*$ is the dual of the Lipschitz maps, a Banach Space. By the appropriate choice of initial condition, and mutation kernel parameter this model unifies both discrete and continuous, pure selection and mutation selection, measure valued and density valued basic consumer resource models. In this paper well-posedness and uniform eventual boundedness under biologically sound assumptions is presented.
\end{abstract}


\end{frontmatter}


\section{Introduction}
In this paper we apply the techniques developed in \cite{JC2} to a version of the basic consumer resource model developed in \cite{ SmithThieme}.
 In \cite{SmithThieme} we are given the system \begin{equation}\label{base} \begin{split} S'  & = \Lambda -DS - \sum_{j=1}^n f_j(S)I_j, \\
I'_j & = f_j(S)I_j -D_jI_j, \qquad j=1, ..., n,
 \end{split} \end{equation}
This can be interpreted as a chemostat model with $n$ species of consumers competing for the limited substrate $S$ or as an epidemic model for the spread of an infectious pathogen that comes in $n$ different strains and converts susceptible hosts  $S$ into hosts $I_j$ infected with strain $j$. In general this is a competition model where there are $n$ consumers, $I_j,$ competing for a resource $S$ which once consumed provides for the increase in the consumer, $I_j.$

 In this paper we extend the above model and use game theory to model \eqref{base} as an evolutionary game. Once the consumer resource model is formulated in the language of evolutionary game theory we use semiflow theory on metric spaces to mathematically model the evolutionary game as a semiflow on a suitable metric space.

As a brief recap, before we begin with the formal definitions of evolutionary game and semiflow for this paper we briefly outline the need for this abstract machinery. We take the following from \cite{JC1}.  We consider the following EG (evolutionary game) model of generalized logistic growth with pure selection (i.e., strategies
replicate themselves exactly and no mutation occurs)  which was developed and
analyzed in \cite{AMFH}:
\begin{equation}
 \frac{d}{dt} x(t,q) = x(t,q) (
q_1 -q_2 X(t)), \label{logiseq}\end{equation}
 where $X(t) = \int_Q x(t,q) dq$ is the total population, $Q \subset \text{int}(\mathbb{R}_+^2)$ is compact
 and the state space is the set of continuous real valued functions
 $C(Q)$. Each $ q=(q_1, q_2) \in Q$ is a two tuple where $q_1$ is an
 intrinsic replication rate and $q_2$ is an intrinsic mortality
 rate. The solution to this model converges to a Dirac
 mass centered at the fittest $q$-class. This is the class with the highest birth to death ratio
 $\frac{q_1}{q_2}$,
 and this convergence is in a topology called $weak^* $
  (point wise convergence of functions) \cite{AMFH}. However, this Dirac limit is not in the
  state space as it is not a continuous function. It is a measure.  Thus, under this formulation one cannot treat this Dirac mass
  as an equilibrium (a constant) solution and hence the study of linear stability analysis is not possible.

  Other examples
  for models developed on classical state spaces such as $L^1(X,\mu)$ that demonstrate the emergence of Dirac measures in the asymptotic limit from smooth initial densities are given in \cite{AMFH,AFT,calsina,CALCAD,GVA,P,GR1,GR2}.
  In particular, how the measures arise naturally in a biological and adaptive dynamics environment is illustrated quite well in \cite[chpt.2]{P}.
   These examples show that the chosen state space for formulating such selection-mutation models must \textbf{contain} densities and Dirac masses and the topology used must \textbf{contain the ability to demonstrate convergence} of densities to Dirac masses.
This process is illustrated in the precursors to this work in \cite{CLEVACK,CLEVACKTHI}.

  In this paper an Evolutionary Game (EG) is defined as a game in which the strategy profiles evolve over time under evolutionary forces (EF) i.e. birth, mortality, mutation, selection (replication), recombination, drift etc... They can also be termed Universal Darwinian Games. Here each consumer is modeled as a strategy and the set of strategies is modeled as a compact metric space, $Q.$ The quantity of the limiting substrate is modeled as a real variable, $S$. The state of the game at a particular time $t$ is modeled as an ordered pair, $ [ S(t),\mu(t)]$  subject to constraints equations. $S$ models the resource and $ \mu \in BL^*= BL(Q)^*$, the dual of bounded  Lipschitz maps on $Q, $  models the distribution of the population of consumers among the strategies.

   However, we also add a nonlinear mutation parameter $ \gamma $ , so that there can also be mutation among the consumers. This is a particularly useful model when the consumers are strains of a pathogen e.g. flu virus where mutation is a fundamental component of its evolution.  This evolutionary game is modeled as a semiflow on a suitable metric space subject to constraints.
\begin{definition}
If $X$ is a metric space, and $ J \subset \R_+$ is an interval that contains zero then a map $$ \Phi: J \times X \rightarrow X $$ is called a local (global autonomous) semiflow if:
\begin{itemize}
\item[(1)] $\Phi(0;x) =x.$
\item[(2)] $\Phi(t+s; x) = \Phi(t; \Phi(s;x))$,  $\forall t, s \in J , ~x \in  X.$
\end{itemize}
\end{definition}

 If $f : X \rightarrow X $ is a locally Lipschitz vectorfield and $ x(t)$ is the unique solution to  $x'(t) = f(x)$  and $ x(0) =x_0$. Then we obtain a global autonomous semiflow $\Phi(t; x_0) = x(t). $ This semiflow is always continuous \cite[ Chpt.1, pg.19]{Thi03}.

 In particular, in the present paper we let $ [ X, D_X ] $ be our metric space where  $$ X = \mathbb{R} \times BL^* \times L(Q;\mathcal{P}^*). $$ Here $Q$ is a compact metric space and $ BL=BL(Q)$ are the bounded Lipschitz maps on $Q.$ $BL^*$ is the norm dual of $BL$ and  $ L(Q;\mathcal{P}^*)$ are the Lipschitz maps into $\mathcal{P}^*$. Elements of $\mathcal{P}^*$ are to be thought of as generalizations of probability measures.  They are elements of $BL^*$ of norm 1. $ \gamma \in L(Q;\mathcal{P}^*)$ is the parameter of our system and is to be thought of as a family of ``probability distributions" indexed by $ Q$. It is the \textbf{mutation kernel}. The metric $ D_X$ satisfies
  $$ D_X( (s_1,u_1, \gamma_1),(s_2, u_2, \gamma_2)) = |s_1 + s_2| + \|u_1 -u_2 \|_{BL}^* + \|\gamma_1 -\gamma_2\|_{\infty}^*. $$


 ( See subsection \ref{technical} for the definitions of  $ \| \cdot \|_{BL}^*$  and $ \|\cdot \|_{\infty}^*.$ )

  In order for a semiflow to model our Evolutionary Game it must satisfy the \textbf{ constraint equations}. In other words our (EG) model is
 an ordered triple $$(Q,\Phi(t;\cdot), \mathcal{F})$$  subject to:

\begin{equation}\label{mconstraint}\frac{d}{dt}\Phi(t;x)[g]= \cF[\Phi(t;x)][g], \text{ for every}
~~g \in BL(Q). \end{equation}
Here $Q$ is the strategy
(compact metric) space, $\Phi(t;x)$
is a semiflow on $X$ and $$\cF : X \rightarrow BL^*$$  is a  vector field (parameter dependent) such that $ \Phi$ and $\cF$
satisfy equation \eqref{mconstraint}. 

 Here is a heurestic understanding of the model in the measure theoretic setting. By this we mean that if $\cal{M}$ denotes the finite signed measures on $Q,$ and $C(Q)$ denotes the continuous functions on $Q$ then we notice that $ \mathcal{M} = C(Q)^* \subset BL^*$ and we give an interpretation of the model in this setting. If $ x = [ s, \mu, \gamma] \in X $, then the real variable $s$ models the amount of resource available, $ \mu (E)$ is  a measure of the quantity of strategies present in the Borel set $ E,$ $ \gamma $ is the mutation kernel. This means that $\gamma(q)(E)$ is the proportion of the $q$-strategy population offspring that are in the Borel set $E$.

   From \cite{SmithThieme} we see that for pure selection the equilibrium point was a dirac mass. The obvious choice for state space was $\mathbb{R}_+ \times \mathcal{M}_+$, under the $weak^*$ topology. Where $\mathcal{M}_+$  denotes the cone of the positive measures. However, $ \mathbb{R}_+ \times \mathcal{M}_+$ is a complete metric space and not a Banach Space.  With slight modifications of the definitions one could use the techniques of either mutational analysis \cite{JPA1,JPA2, Lorenz} or differential equations in metric spaces \cite{Tabor02} or arcflows of arcfields \cite{CC,CB}  to generate a semiflow that satisfies the equivalent of the initial value problem in semiflow theory language.

The method employed here is  that we find a Banach Space,  $\mathbb{R}\times BL^*$ containing $ \mathbb{R}_+\times\mathcal{M}_+ $ as a closed metric subspace. Then we extend the constraint equation on $ \mathbb{R}_+\times\mathcal{M}_+ $ to one on $\mathbb{R}\times BL^*$. The semiflow resulting from the solution of the generalized constraint equation has $ \mathbb{R}_+\times\mathcal{M}_+ $ as a forward invariant subset and hence we generate our semiflow on $ \mathbb{R}_+\times\mathcal{M}_+ $. This is essentially the method employed here. However, using this approach we see that we generate a semiflow on any forward invariant subset of $X$.

So we see from the above that using measures we can capture a lot of the needed elements for developing a semiflow model of the consumer resource evolutionary game. However, using the measures one runs into long and less than efficient arguments. For example, one must use different topologies on the measures to derive key results. One uses total variation on a fixed point space to obtain the model, and then one places another topology, weak star to develop continuity of the model. Some of the arguments and estimates are long and cumbersome. However, using the machinery developed in \cite{JC2}, in particular the multiplicaton $ \bullet $ defined therein, we can use $ \R \times BL^* $ as a statespace and cut down considerably on complexity of arguments and obtain stronger results. Furthermore, the model developed on  $\R \times BL^*$ includes all the other aforementioned cases as special cases. So this model is the most general developed so far and the mathematics is much nicer since we can use a \emph{norm} for estimates as opposed to \emph{families of seminorms}.

This paper is organized as follows. Section 1 comprises a brief description of the paper along with motivation and a brief literature review. In section 2 we develop the constraint equation and in section 3 we give background mathematical definitions and notation needed to follow the later material. In section 4 the consumer resource model is developed. Section 5 is devoted to showing wellposedness and non negativity. Section 6 shows uniform eventual boundedness. Section 7 demonstrates the unifying power of the model and shows that it encompasses the model in \cite{SmithThieme}. Section 8 is a conclusion which includes some future problems.

  \section{Constraint Equation}\label{CE}

 We start with the discrete pure selection heterogeneous parameter density consumer resource model developed in \cite{SmithThieme}. From this model we add resource dependence to all the vital rates $ B, D $ and add a nonlinear mutation term with a mutation kernel $\gamma .$ We then integrate and use Fubini to obtain a measure theoretic model. Hence the discrete model in \cite{SmithThieme}

\begin{equation}\label{x}\begin{split}
S' &=  \Lambda -DS - \sum_{j=1}^{n} B_j(S)I_j, \\
I'_j & = B_j(S)I_j -D_jI_j, \qquad \mbox{ j = 1, ... , n,}
\end{split}
\end{equation}
becomes the model

\begin{equation}\label{xx}\begin{split}
S'(t) &=    \Lambda -DS - \int_Q B(S(t),q)\mu(t)(dq)\\
\mu'(t)(E) & =\int_Q B(S(t),q)\gamma(q)(E)\mu(t)(dq)-\int_E D(S(t),q)\mu(t)(dq).
\end{split}
\end{equation}

These models  can be interpreted as a chemostat model where the different species are  consumer strategies $q$  competing for the limited substrate $S$ or as an epidemic model for the spread of an infectious pathogen that comes in strains ($q$ - strategies) and converts the resource $S$ into hosts in $E$ where $E$ is a Borel subset of $Q$ a compact metric space.

With these interpretations, above $D$ is the dilution or washout rate of the substrate or resource or death rate of the host. $\Lambda$ is the rate at which fresh substrate or resource is entering the system and $\gamma(  q)(E)$ is the proportion of the $ q$ -strategy population offspring adopting strategies that are in $E$, a Borel subset of $Q$.

Starting from \eqref{xx} we apply the  $\bullet$  operation defined in \cite{JC2} (and redefined below in section \ref{PM}) to obtain:

\begin{equation}\begin{split}\label{xxx}
S'(t) &=    \Lambda -DS -  B(S(t), \cdot)\bullet \mu(t)[1]\\
\mu'(t) & =B(S(t), \cdot)\gamma(\cdot)\bullet\mu(t)- D(S(t), \cdot)\bullet \mu(t).
\end{split}
\end{equation}

Here $S$ is a real function of a real variable $t$ and $\mu$ is a $BL^*$ valued function of a real variable $t$.

Suppose  $ [S,\mu ]$ is a solution to \eqref{xxx}. Then define
  $$ \Phi(t;(s,u, \gamma))= [ \mu(t,s,u,\gamma), \gamma] $$  where $$ \frac{d\Phi}{dt}(t;(s, u, \gamma))= [\frac{d\mu}{dt}(t,s,u,\gamma), \gamma]. $$

 Let \begin{equation} \label{eq:F}F(m,\gamma) = F( S,\mu, \gamma) = [F_1( S,\mu, \gamma), F_2( S,\mu, \gamma)] \end{equation}
where

 \begin{equation}\label{yy}\begin{split}
 F_1(m,\gamma ) &= \Lambda -D S +  B(S,\cdot) \bullet \mu[1] \\
  F_2(m, \gamma)&= \gamma( \cdot ) B(S,\cdot) \bullet \mu -  D(S,\cdot)\bullet {\mu}
  .\end{split}
\end{equation}

  If
 $ \cF( \mu, \gamma) = [F(\mu, \gamma), \gamma] $ where $F$ is as in \eqref{eq:F}, then

 \begin{equation*} \begin{split}
 \cF[ \Phi(t;(s,u, \gamma))]  & = \cF [ \mu(t,s,u, \gamma), \gamma] =  [F(\mu(t,s, u, \gamma), \gamma] =[\frac{ d\mu}{dt}(t,s, u,\gamma), \gamma]
   = \frac{d\Phi}{dt}(t; (s, u, \gamma))
 \end{split} \end{equation*}
 or for  $ x =(s, u, \gamma) $
 \begin{equation}\label{CONSTRAINT} \begin{split}    \frac{d\Phi}{dt}(t;x)& = \cF[ \Phi(t; x)] \mbox { ~ and } \\
  \Phi(0;x) = x \end{split} \end{equation}

 \eqref{CONSTRAINT}  is the $BL^*$ valued constraint equation.

\section{Preliminary  Material and the definition of $\bullet$}\label{PM}

We begin modeling with  $([Q,d], \cB(Q),P)$ where $[Q,d]$ is a compact metric space, $\mathcal{B}
(Q)$ are the Borel sets on $[Q,d]$ and $P$ is a probability measure on the Measurable Space $([Q,d], \mathcal{B}(Q))$ representing an initial weighting on the strategies. One can think of $Q$ as a compact subset of $\mathbb{R}^n$ and $P$ as a probability measure (initial weighting) on this set.
$Q$ above is used to model the space of strategies. What we seek as a model of our game is a semiflow subject to the constraint equation \eqref{CONSTRAINT} which will follow easily from a parameter indexed family of solutions to \eqref{xx} above.

\subsection{Birth and Mortality Rates}

 As mentioned in the second paragraph, the evolutionary forces that act on our population are: $ B(S,q)$, $D(S,q)$, $\gamma(q).$
 $B(S,q), ~D(S,q)$ are the per capita uptake and washout rates of the $ q- strategy $ populations respectively.  We assume the following regularity.

\begin{itemize}
\item[(A1)] $ B( \cdot,q)$ is nondecreasing and Lipschitz continuous uniformly for $q \in Q.$ $ \forall q \in Q, B(0,q)=0 $ and $B(S,q)>0$ for $ S>0.$
\item[(A2)]$  D(\cdot,q)$ is nonincreasing and  Lipschitz  continuous uniformly for $ q \in Q.$ Also $ \min_{ q,S} D(S,q)=\varpi > 0 .$ (This means that there is some inherent mortality not density
related)
\end{itemize}
These assumptions are of sufficient generality to capture many nonlinearities of classical population dynamics including Ricker,
Beverton-Holt, and Logistic (e.g.,  see \cite{AFT}).

 \subsection{Technical Preliminaries for Bounded Lipschitz Formulation } \label{technical}


If $ [Y, \|\cdot\|_Y] $ is  a Banach Algebra, $ C(Q;Y)$ denotes the continuous $Y$- valued maps under the uniform norm,
 $$  \|f\|_{\infty}   = \sup_{q \in Q} \|f \|_{Y} . $$

Two important subspaces are  $$ L(Q;Y)[M] \subset L(Q;Y) \subset C(Q;Y).$$
Where $L(Q;Y)$ is the dense subspace of all $Y$ -valued   Lipschitz maps and $L(Q;Y)[M]$ is the locally compact subspace of Lipschitz maps with Lipschitz bound smaller than or equal to $M$.

 If no range space is specified then $ C(Q) =C(Q;\mathbb{R})$, denotes the Banach space of continuous real valued functions on $Q$.
 The two important subspaces mentioned above are then denoted as $L$ and $L[M]$ respectively.

 $L$ also has a finer structure. Indeed, if $ f \in L, $ define

$$ \|f\|_{Lip}= \sup \left\{ \frac{\|(f(x)-f(y))\|_Y}{d(x,y)}: x,y \in Q, x \neq y \right\}.$$

  Under the  norm $\| f \|_{BL} = \|f\|_{\infty}+ \|f\|_{Lip}$, $L$ becomes a Banach space denoted  $[BL, \|\cdot \|_{BL}]. $

  $ [BL^*, \|\cdot  \|_{BL}^*] $ denotes the continuous dual of this Banach Space and it has a closed convex subspace \begin{equation}\cP ^* = \{ \mu \in BL^*_+ ~| ~ \|\mu\|_{BL}^* = 1 \} . \begin{footnote}{ If $S$ is a subset of a Banach space, then $S_+$ is the intersection of $S$ with the positive cone. } \end{footnote} \end{equation}

 \begin{remark} $L$ and $BL$ are the same set, the topology is just different.
\end{remark}

Crucial to the success of our modeling efforts is the forming of the parameter space,  $ L(Q;\cP^*) \subset C(Q; BL^*),$ which models the mutation kernel.  It is a convex subset of   $C(Q; BL^*).$ \\

\emph{Some Algebra :} \\

Firstly we note that both $[C(Q;Y), \|\cdot \|_{\infty}] $ and  $ [BL(Q;Y),\|\cdot \|_{BL}] $ are also Banach Algebras and we have the inequality \begin{equation} \label{BA} \|fg \| \le \|f \| \|g \| \end{equation} holding in each space.

Secondly,  we  view $ \gamma \in L(Q; BL^*)$   as a family of bounded linear functionals indexed by $Q$. It has properties that need elucidating for our modeling purposes.  $ L(Q; BL^*)$ is a unital BL- module. Indeed if $f,~ g\in BL,$  $ \gamma \in L(Q;BL^*) $
 \begin{equation} \label{action0} (f \cdot \gamma) (q)[g] =f(q)\gamma(q)[g] \mbox{ and } \| f\gamma \|^*_{\infty}\le \|f\|_{\infty} \|\gamma\|_{\infty}^* \begin{footnote}{ If $ \gamma \in C(Q;BL^*)$, $ \| \gamma\|_{\infty}^* = \sup_{q \in Q} \|\gamma(q)\|_{BL}^* $} \end{footnote}.\end{equation}  We will denote this action simply as $ f\gamma$ since it is just pointwise multiplication. So one can multiply a family of functionals by a Lipschitz map and obtain another family of functionals. Moreover,
the new \textbf{uniform} normed product is no larger than the \textbf{uniformed} product of the norms.

Thirdly,
 $$ L \hookrightarrow L(Q;BL^*)  \mbox { by }   f \mapsto f(\cdot)\delta_{(\cdot)} $$  is an isometry. Where $ \delta_{(\cdot)}$ is the delta functional.

 This allows us to view a Lipschitz function, $f$,  as a \textbf{family} of bounded linear functionals on $ BL$  indexed by $ Q.$ Moreover this viewing preserves the uniform norm, i.e.
  \begin{equation*} \|f\|_{\infty} = \|f(\cdot)\delta_{(\cdot)}\|_{\infty}^*. \end{equation*}

 Fourthly, we need to ``multiply" a functional by a family of functionals.   Let $ M^*_b:= [M_{b}^*(BL;\R), \|\cdot\|^*_{BL}], $ denote the normed  $\mathbb{R}$ -Algebra of bounded maps of $BL$  into $\R$ where we have pointwise addition and multiplication and the norm defined as

  $$ \|\mu \|_{BL}^* = \sup_ {  g \in BL, g \ne 0}   \frac{  | \mu(g)|  }{ \|g\|_{BL}}  $$  

 If $$ \Sigma = [BL (Q; M^*_b), \|\cdot\|_{BL} ]$$  then $\Sigma $ is a $\mathbb{R}$- Algebra  under pointwise addition and multiplication and  $ M^*_b(BL;\R)$ is a $\Sigma$- module. Indeed,  under the action
  $$ \bullet: \Sigma \times M^*_b(BL;\R) \rightarrow M^*_b(BL;\R)$$ given by

 $$ (\gamma\bullet \mu) [g] = \mu \Bigl[ \gamma(\cdot)[g] \Bigr] ~~ \forall g \in BL, ~ \forall \gamma \in \Sigma  $$ we have an action.
 This is a  bounded Lipschitz   functional since $ \forall g \in BL, \gamma(\cdot)[g]$ is bounded and Lipschitz since $ \gamma \in BL(Q;M^*_{b}).$
 With respect to the normed  product we have \begin{equation} \label{action21} \|\gamma \bullet \mu\|_{BL}^* \le \|\gamma\|_{BL} \|\mu \|_{BL}^* .\end{equation}
  Moreover, if $ \mu \in BL^*_{+}$, \eqref{action21} becomes
  \begin{equation} \label{action21b} \|\gamma \bullet \mu\|_{BL}^* \le \|\gamma\|^*_{\infty} \|\mu \|_{BL}^* .\end{equation} where $$ \| \gamma\|^*_{\infty} = \sup_{ q \in Q} \|\gamma(q)\|_{BL}^*.$$

  $ \bullet$ above allows us to ``multiply" a functional,
   $ \mu \in M_b^*$, by a family of functionals $ \gamma \in \Sigma .$

   This new multiplication gives us some important information about our mutation parameter space $L(Q;BL^*).$

    Indeed,

    \begin{itemize}
    \item[(1)] First notice   $$ L(Q;BL^*) \times  BL^* \subset \Sigma \times M_b^*(BL; \R).$$ If we think of  $L(Q;BL^*)$  as $ [ BL(Q;BL^*), \|\cdot\|_{BL}]$ (same set different topology), then we actually have that $$ \bullet:  BL(Q;BL^*)\times BL^* \rightarrow BL^* $$ by
 \begin{equation}\label{action2} (\gamma \bullet \mu)[g]=\mu\Bigl[\gamma(\cdot)[g]\Bigr] .\end{equation}
  The $\bullet$ operation does \textbf{not }make $BL^* $ into a $BL(Q;BL^*)$- module since $BL(Q;BL^*)$ is not a ring . However, this restriction of $\bullet$ is \textbf{bilinear}.
   \item[(2)] Also note that if  $f \in BL$, then $f\bullet \mu $ is well defined as well. Indeed, from the thirdly observation in the \emph{Some Algebra} section we view $f$ as the family $\gamma(q) =f(q) \delta_q$, and
\begin{equation} \label{identity} (f\bullet \mu)[g] = (f  \delta )\bullet \mu[g] = \mu \Bigl [ f(\cdot)\delta_{( \cdot)}[g]\Bigr ] = \mu  [ f(\cdot) g( \cdot)] = \mu[fg]. \end{equation}
Furthermore \begin{equation} \label{action11} \|f \bullet \mu\|_{BL}^* \le \|f\|_{BL} \|\mu \|_{BL}^* .\end{equation}

\item[(4)] In all cases $ \bullet $ behaves nicely with respect to norm estimation in all norms. The normed product is no larger than the product of the norms.
\end{itemize}

\emph{Miscellaneous:} \\

 If $ \nu \in BL^*$, $$ B_{a}[\nu] = \{ \mu \in BL^* ~|~~ \|\mu -\nu \|_{BL}^* < a \} .$$

 Below $ L\cP^*[M] = L(Q;\mathcal{P}^*)[M]$ and likewise for $BL\cP^*[M] .$

 $\textbf{0}$ denote the zero functional and $1$ denotes the constant function that takes the value $1$.

For any time dependent mapping, $f(t)$,  we let $f' (t)=\frac{df}{dt} (t)$

\section{ Main Well-Posedness Theorem} \label{WP}

The following is the main theorem of this section.
\begin{theorem}\label{main}   Let  $X = \mathbb{R}\times {BL^*}  \times L(Q;\mathcal{P}^*).$ Then $[X,D_X] $ is a metric space where
 $$ D_X( (s_1, u_1, \gamma_1),(s_2,  u_2, \gamma_2)) = |s_1-s_2| + \|u_1 -u_2 \|_{BL}^* + \|\gamma_1 -\gamma_2\|_{\infty}^*. $$

Moreover there exists a global autonomous semiflow where
$$ \Phi:{\mathbb {R}_+} \times  X  \to X $$  satisfying the following:
\begin{enumerate}
\item  There exists a  continuous mapping $$ \varphi : \R_+ \times \R\times {BL^*}  \times BL(Q;\mathcal{P}^*) \rightarrow \R \times BL^* $$ such that
 $$ \Phi(t; (s,u, \gamma)) = (\varphi(t,s,u,\gamma), \gamma).$$

\item For fixed $(s,u, \gamma ) \in \R\times BL^* \times BL(Q;\mathcal{P}^*)$, the mapping $t \mapsto \varphi(t, s, u,\gamma) \in C( \R_+; \R\times BL^*) $ is the unique \emph{solution} to

\begin{equation} \left\{   \begin{split}
S'(t) &=    \Lambda -DS -  B(S(t),q)\bullet \mu(t)[1]\\
\mu'(t) & =B(S(t),\cdot)\gamma(\cdot)\bullet\mu(t)- D(S(t), \cdot)\bullet \mu(t)\\
[S(0),\mu(0)] & =[ s_0,\mu_0]
\end{split} \right.
\end{equation}

Moreover, if $ F$ is as in  \eqref{eq:F} and  $$ \cF : X \rightarrow X $$ by
$$  \cF[ \mu, \gamma] =[ F(\mu, \gamma), \gamma] \mbox{ ~~ and ~~ } \Phi'(t;(s,u,\gamma)) = [ \varphi'(t,s,u, \gamma), \gamma] $$

Then \begin {equation} \left\{\begin{array}{ll}
 \displaystyle {\Phi'}(t;x) & =  \cF [ \Phi(t;x)]    \\
\Phi(0;x)=x.
\end{array}\right.\end{equation}

\item   If   $ X_+ = \mathbb{R}_+\times  BL^*_+ \times L(Q; \mathcal{P}^*)$, then $ X_+$  is forward invariant under $\Phi $ i.e. $\Phi(t; X_+) \subset X_+ ,$ $\forall t \in \R_+ $.

\item  $ \forall N \in \mathbb{N},$ if   $ X_N = [0,N] \times[-N,N]\times  B_N[\textbf{0}]_+ \times L\cP^*[N] ,$  then
 $\Phi $ is Lipschitz continuous on $X_N.$

\end{enumerate}
\end{theorem}

We now establish a few results that are needed to prove
Theorem \ref{main}.

\subsection{Local Existence and Uniqueness of Dynamical System}\label{WP1}

\emph{Truncation:}\\

With this background we prepare to obtain the semiflow that will model our evolutionary game. If $ F$ is the vectorfield defined in \eqref{eq:F} then for each $ N \in \mathbb{N}$, define $F_{N}$ as follows.
If $j $ is one of the functions $ B,D$  then we extend $j$ to $\R \times Q$ by setting $j_{N}(x,q) = j(0,q) $ for $x \le 0$ and
  $ j_{N}(x,q) = j(N,q) $ for $x \ge N
 $. Then $ {j}_{N}: \R \times Q \to \R_+$ is bounded and Lipschitz
continuous.  Let $ {F}_{N} (m, \gamma)
$ be the redefined vector field obtained by replacing $j$ with
$j_{N}.$


For each $(s,u, \gamma) \in \mathbb{R}\times BL^* \times BL(Q;\mathcal{P}^*)$,  we will resolve the following IVP first.

 \begin {equation} \left\{\begin{array}{ll}\label{M2}
  m'(t,s,u, \gamma) =  F_{N}(m,\gamma)  \\
m(0,s,u,\gamma)= u.
\end{array}\right.\end{equation}

Here  let \begin{equation} \label{eq:FT}F_{N}(m,\gamma) = F_N( s,\mu, \gamma) = [F_{N1}( s,\mu, \gamma), F_{N2}( s,\mu, \gamma)] \end{equation}
where

 \begin{equation}\label{yyT}\begin{split}
 F_{N1}(m,\gamma ) &= \underbrace {\Lambda +  B(s, \cdot) \bullet \mu[1]}_{\mbox{ $F_{N11}$}} - \underbrace{ D s}_{\mbox{ $F_{N12}$}}   \\
  F_{N2}(m, \gamma)&= \underbrace{\gamma( \cdot ) B({s}, \cdot) \bullet \mu}_{ \mbox{$ F_{N21}$}} -  \underbrace{ D(s, \cdot)\bullet {\mu}}_{ \mbox{$F_{N22}$ }}
  .\end{split}
\end{equation}

\begin{lemma} \label{LF}(Lipschitz $F_N$)

\begin{itemize}
\item[(i)] $\forall N \in \mathbb{N}$, there exists continuous $${F}_N :\mathbb{R}\times  BL^* \times BL(Q; \mathcal{P}^{*}) \rightarrow \mathbb{R}\times BL^*.$$

     \item[(ii)] $\forall a>0,$ $ \forall M >0 $, if $${F}_N:[-a,a] \times  \overline{ B_{a}[\textbf{0}]  } \times BL\cP^{*}[M] \rightarrow \mathbb{R}\times BL^* $$ or $${F}_N:[-a,a] \times  \overline{ B_{a}[\textbf{0}]  }_+ \times L\cP^{*}[M] \rightarrow \mathbb{R}\times BL^* $$ then $F_N$ is bounded and Lipschitz.
\end{itemize}
\end{lemma}

\begin{proof}
First notice that $(i)$ follows from  $(ii)$ since

\begin{equation} \R\times BL^* \times  BL(Q; \mathcal{P}^*) = \cup_{N , M \in \mathbb{N}}[-N,N]\times BL^* \cap \overline{{B_{N}[\textbf{0}]}} \times BL\cP^{*}[M]\begin{footnote}{ See \eqref{eq:union}} \end{footnote} \end{equation}

and   $$  \overline{{B_{N}[\textbf{0}]}} \subset  \overline{{B_{N+1}[\textbf{0}]}} \mbox{~ , ~} BL\cP^*[M] \subset BL\cP^*[M+1] .$$

We will prove the second condition in $(ii).$ The first is straightforward and the only real difference in the argument used below is that one uses the estimate in \ref{action21} instead of the estimate in \ref{action21b}. Let $ (s, \mu, \gamma)$ and  $(r, \nu, \lambda )$ be two points in $  [-a,a]\times \overline{ B_{a}[\textbf{0}]  }_+ \times L\cP^{*}[M].$ We  must find $ B_S, B_{\mu}, B_{\gamma} $ such that $$ \|{F}_N( s, \mu, \gamma) -{F}_N(r, \nu, \lambda)\|_{\R  \times BL^*} \leq B_S|s-r| + B_\mu \| \mu- \nu\|_{BL}^* +  B_{\gamma} \| \gamma- \lambda \|^*_{\infty} .$$

\begin{multline}
\|{F}_N(s,\mu, \gamma) - {F}_N(r,\nu, \lambda)\|_{\R\times BL^*}  =  ~ \\
 \big |{F}_{N1}(s, \mu, \gamma)-{F}_{N1}( r,\nu, \lambda)\big | + \big \| {F}_{N2}(s,\mu, \gamma)-{F}_{N2}( r,\nu, \lambda)\big \|_{BL}^*
\end{multline}
where
$$ F_{N1}(s, \mu,\gamma ) = \Lambda -D s +  B_{N}(s,\cdot) \bullet \mu[1] $$
$$ F_{N1}(r, \nu,\lambda ) = \Lambda -D r +  B_{N}(r,\cdot)\bullet \nu[1] .$$

Hence, \\ $ F_{N1}(s, \mu,\gamma )-F_{N1}(r, \nu,\lambda )= $  $D( r-s) + [B_{N}(s,\cdot)-B_{N}(r, \cdot)]\bullet \mu[1] +  B_{N}(r, \cdot)\bullet(\mu-\nu)[1]$

and
\begin{equation}\label{FN1} |F_{N1}(s, \mu,\gamma )-F_{N1}(r, \nu,\gamma )| \le   (D + \|B_N\|_{Lip} \|\mu\|_{BL}^* )| s-r| + \|B_{N}\|_{BL} \| \mu - \nu\|_{BL}^*.\end{equation}

For $F_{N2}$, 

 $$F_{N2}(s, \mu, \gamma)= \gamma( \cdot )B_{N}(s,\cdot)\bullet \mu -  D_{N}(s,\cdot)\bullet {\mu},$$
 $$ F_{N2}(r,\nu, \lambda)= \lambda( \cdot )B_{N}(r,\cdot ) \bullet \nu -  D_{N}(r,\cdot) \bullet {\nu}$$

and  \begin{equation}\begin{split}
F_{N2}(s, \mu, \gamma)- F_{N2}(r,\nu, \gamma) & = ( \gamma - \lambda)( \cdot) B_N(s,\cdot) \bullet \mu   ~ + \\
  \lambda( \cdot)(B_{N}(s,\cdot)- B_{N}(r,\cdot))\bullet \mu & +   \gamma( \cdot)B_{N}(r,\cdot)(\mu- \nu) ~ + \\
      (D_{N}(s,\cdot)-D_{N}(r,\cdot)) \bullet {\mu} & +   D_{N}(r,\cdot) \bullet (\mu -\nu).
     \end{split} \end{equation}

\noindent Hence,
\begin{multline*} \|F_{N2}(s, \mu, \gamma)- F_{N2}(r,\nu, \gamma)\|_{BL}^* \le  \|\gamma - \lambda \|^*_{\infty} \|B_N \|_{\infty} \|\mu \|_{BL}^* ~ + \\
 \| \lambda \|^*_{\infty} \|B_{N}\|_{Lip}\|\mu\|_{BL}^* |s-r| ~  + ~  \| \gamma \|_{BL} \|B_{N}\|_{BL}\|\mu-\nu\|_{BL}^* ~ + ~ \|D_{N}\|_{Lip}\| \mu \|_{BL}^* |s-r| ~ +~  \\ \|D_{N}\|_{BL}\|\mu-\nu\|_{BL}^* \end{multline*}

and \begin{multline} \label{FN2}\|F_{N2}(s, \mu, \gamma)- F_{N2}(r,\nu, \gamma)\|_{BL}^*  \le
\Bigl (\| \lambda \|^*_{\infty} \|B_{N}\|_{Lip}\|\mu\|_{BL}^*   +   \|D_{N}\|_{Lip}\| \mu \|_{BL}^* \Bigr)|s-r| ~ + \\  \Bigl ( \| \gamma \|_{BL} \|B_{N}\|_{BL} + \|D_{N}\|_{BL} \Bigr )\|\mu - \nu\|_{BL}^* + \|\gamma - \lambda \|^*_{\infty} \|B_N \|_{\infty} \|\mu \|_{BL}^*.
\end{multline}

 Hence, \begin{equation} \begin{split}\label{FS}
  \|{F}_N(s,\mu, \gamma) - {F}_N(r,\nu, \lambda)\|_{\R\times BL^*} \le  & |s-r|\Bigl (\| \lambda \|^*_{\infty} \|B_{N}\|_{Lip}\|\mu\|_{BL}^*   +   \|D_{N}\|_{Lip}\| \mu \|_{BL}^*  \\
    +~  D ~  + ~  \|B_N\|_{Lip} \|\mu\|_{BL}^*\Bigr ) + & \Bigl ( \| \gamma \|_{BL} \|B_{N}\|_{BL} + \|D_{N}\|_{BL} + \|B_{N}\|_{BL} \Bigr )\|\mu - \nu\|_{BL}^*  \\
    +~  \|\gamma - \lambda \|^*_{\infty} \|B_N \|_{\infty} \|\mu \|_{BL}^*. \end{split} \end{equation}

So if $C_W$ is such that $ |s| +\|\mu\|_{BL}^* + \| \gamma \|^*_{\infty} \leq C_W$ for $(s,\mu, \gamma)
\in [-a,a] \times  \overline{ B_{a}[\textbf{0}]_+  } \times L\cP^{*}[M] $ 
our result is immediate.

\end{proof}


\begin {lemma}\label{E}(Estimates)  Let  $T>0.$  If $ \zeta, \beta \in C( [0,T];BL^*)  $ and  $ s,t \in  [0,T]$  we have the following estimates:
\begin{enumerate}
%
\item
\begin{itemize}
\item[(a)] \mbox{ As a function of } $q$,
\begin{align}
 \|e^{-\int_{s}^{t}{D}_{N}(\zeta(\tau)(1),q)d\tau}\|_{Lip} & \le \|D_N(\cdot,\cdot)\|_{Lip} T~, & \|e^{-\int_{s}^{t}{D}_{N}(\zeta(\tau)(1),q)d\tau}\|_{\infty} \le 1.
\end{align}
\item[(b)] If  \begin{equation} \begin{split} F(q) &  = e^{-\int_{s}^{t}{D}_{N}(\zeta(\tau)(1),q)d\tau}-e^{-\int_{s}^{t}{D}_{N}(\beta(\tau)(1),q)d\tau} \\
   \|F\|_{\infty}  & \le \|{D}_{N}(\cdot,\cdot)\|_{BL}\int_{s}^{t}\| \zeta(\tau)- \beta(\tau)\|^*_{BL}d\tau .\end{split} \end{equation}
 \end{itemize}
\end{enumerate}
\end{lemma}

\begin{proof}

  \begin{itemize}
\item[(a)]
  Using the mean value theorem on the $C^{\infty}( \mathbb{R})$ function, $e^x$, there exists $\theta(s,t)>0$ such that $$ \begin{array}{lll} && |e^{-\int_{s}^{t} {D}_{N}(\zeta(\tau)(1), \hat q) d\tau}
  - e^{-\int_{s}^{t} {D}_N(\zeta(\tau)(1), q) d\tau} | \\
& \leq &  e^{-\theta} |\int_{s}^{t}
\bigl[{ D}_N(\zeta(\tau)(1), \hat q) -
 {D}_N( \zeta(\tau)(1), q) \bigr] d\tau \bigr | \\
& \leq &    \|{D}_N(\cdot,\cdot)\|_{Lip}T d(\hat q, q).  \\
 \end{array}$$

\item[(b)] Using the mean value theorem on the $C^{\infty}( \mathbb{R})$ function, $e^x$, there exists $ \theta = \theta(s,t) > 0$, such that

 $$ \begin{array}{lll} |F(q)| & = & |e^{-\int_{s}^{t} {D}_{N}(\zeta(\tau)(1),q) d\tau} - e^{-\int_{s}^{t} {D}_N(\beta(\tau)(1),q) d\tau} | \\
& \leq &  e^{-\theta} |\int_{s}^{t}
\bigl[{ D}_N(\zeta(\tau)(1), q) -
 {D}_N(\beta(\tau)(1), q) \bigr] d\tau  | \\
 & \leq &  \int_{s}^{t} \|{D}_N(\cdot,\cdot)\|_{BL}\| \zeta(\tau)- \beta(\tau)\|_{BL}^*d\tau
  .\end{array}$$

\end{itemize}

\end{proof}


\begin{proposition} \label{local}If $  T, M >0 $, $ N \in \mathbb{N}$ let  $ F_N$ be as in \eqref{yyT}. There exists  continuous $$\varphi _{NM} : [0,T] \times [-N,N] \times B_N[\textbf{0}] \times L\cP^*[M] \rightarrow BL^*$$  satisfying:

 \begin{enumerate}

 \item For each $ (s,u, \gamma) \in  [-N,N] \times B_N[\textbf{0}] \times BL\cP^*[M] $,  $ t \mapsto  \varphi_{NM}(t,s,u,\gamma) $, is the unique solution to

 \begin {equation} \left\{\begin{array}{ll}\label{eq:M2}
  m'(t) =  F_{N}(m(t),\gamma)  \\
m(0)= u.
\end{array}\right.\end{equation}
in  $ C([0,T]; \mathbb{R} \times BL^*).$
\item $ \varphi _{NM} \bigl( [0,T]\times [0,N] \times B_N[\textbf{0}]_+ \times L\cP^*[M] \bigr )\subset \mathbb{R}_+\times BL^*_{+} $

\item If  $\varphi_{NM}: [0,T] \times [-N,N]\times B_N[\textbf{0}]_+ \times L\cP^*[M] \rightarrow \mathbb{R} \times BL^*$  then  $\varphi_{NM}$ is  Lipschitz continuous.

\end{enumerate}

  \end{proposition}

\begin{proof}   For $ w \in W = C([0,T] ; \mathbb{R}\times BL^*)$ and $\lambda >0 $, define $$ \|w \|_{\lambda} = \sup_{t \in [0,T]} e^{-\lambda t} \|w(t)\|_{BL}^*.$$ It is an exercise to show that $[W, \| \cdot \|_{\lambda}]$ is a Banach space. In fact $ \|\cdot\|_{\infty}$ and
$ \|\cdot\|_{\lambda}$ are equivalent.\\

\emph{Unique local solution to  \eqref{eq:M2}:}\\

 Using  standard techniques for locally Lipschitz vector fields with a parameter into a Banach space, Lemma \ref{LF} relays that we have a unique solution to \eqref{eq:M2} on $ [0,T]$ for any $(s,u, \gamma) \in  [-N,N] \times B_N[\textbf{0}] \times BL\cP^*[M] .$ We can use a Lipschitz argument similar to the one below to show that this mapping is indeed Lipschitz.

 We label this solution $ \varphi_{NM}( \cdot) \equiv \varphi_{NM}(\cdot ,s,u, \gamma)$ (to denote the dependence on  $(s,u, \gamma)$ ). \\


\emph{ Forward invariance of $[0,N] \times B_{N}[0]_+ \times L\cP^*[M]$:} \\

Let $ (s, u, \gamma) \in [0,N] \times B_{N}[0]_+ \times L\cP^*[M], $ if  $ [W, \|\cdot\|_{\lambda}]$ is as above define

$$ W_{\widehat{N}} = \{ \zeta \in W ~ | ~ \zeta( [0,T]) \subset \mathbb{R} \times \overline{B_{\widehat{N}}[0]} \}  \mbox{ where } \widehat{N}>  s + N +(\|F_{N11}\|_{\infty} + \|F_{N21}\|_{\infty})T . $$ Obviously $ W_{\widehat{N}}$ is a closed subspace of $W$ and hence is a complete metric space.

Recall  \eqref{yyT} \begin{equation*} \begin{split}
 F_{N1}(s,\mu,\gamma ) &= \underbrace {\Lambda -  B(s,\cdot) \bullet \mu[1]}_{\mbox{ $F_{N11}$}} - \underbrace{ D s}_{\mbox{ $F_{N12}$}}   \\
  F_{N2}(s, \mu, \gamma)&= \underbrace{\gamma( \cdot ) B(s, \cdot) \bullet \mu}_{ \mbox{$ F_{N21}$}} -  \underbrace{ D(s, \cdot)\bullet {\mu}}_{ \mbox{$F_{N22}$ }}
  .\end{split}
\end{equation*}
If $ \zeta =[\zeta_S, \zeta_{\mu}] \in W_{\widehat{N}} $ define

\begin{equation}\label{irep} \begin{split}
(T \zeta)(t) & = \Bigl[ e^{-Dt}s, e^{-\int _{0}^{t}
D_{N}(\zeta_S(\tau)(1), \cdot)d\tau} \bullet u \Bigr] + \int_{0}^t \Bigl [e^{-D(t-s)}F_{N11}[ \zeta(s),\gamma],
 e^{-\int _{s}^{t} D_{N}(\zeta_S(\tau)(1), \cdot)d\tau}\bullet  \\ F_{N21}[\zeta(s), \gamma]\Bigr ] ds .
\end{split}
\end{equation}

\emph{Contraction Mapping : }\\

From our choice of $ (s,u, \gamma)$ and $\widehat{N}$ ,

$$ T :W_{\widehat{N}} \rightarrow W_{\widehat{N}}.$$

Indeed, if $ \zeta \in W_{\widehat{N}}$, then obviously  $T\zeta$  is continuous in $t$.  Furthermore since $F_{N21}[\cdot, \cdot], F_{N11}[ \cdot, \cdot]$  has the same properties as $ F_N $, namely being uniformly bounded and Lipschitz  on  $[0,N]\times B_{\widehat{N}}[\textbf{0}]_+ \times L\cP^*[M]$, we can use \eqref{action21b}, Lemma \ref{LF}  and Lemma \ref{E} to obtain

\begin{equation} \begin{split}
\|(T \zeta)(t)\|_{BL}^*  & \le s +  \| u\|_{BL}^* +  T ( \|F_{N11} \|_{\infty} + \|F_{N21} \|_{\infty}).
\end{split}
\end{equation}
 Hence $T$ is indeed a mapping from  $W_{\widehat{N}}$ into $W_{\widehat{N}}.$

Moreover for the above choice of $ (s,u, \gamma) $, $T$ is a contraction mapping. Indeed, first notice that since
$ u \in BL^*_+, $ if $ g \in BL(Q),~ \|g\|_{BL} \le 1 , $
 \begin{equation} \begin{split}
 ( e^{-\int _{0}^{t}D_{N}(\zeta_S(\tau)(1), \cdot)d\tau} - e^{-\int _{0}^{t} D_{N}(\beta_S(\tau)(1), \cdot)d\tau})\bullet u [g]
  & =  u [ ( e^{-\int _{0}^{t}D_{N}( \zeta_S(\tau)(1), \cdot)d\tau} - e^{-\int _{0}^{t}
D_{N}( \beta_S(\tau)(1), \cdot)d\tau})g(\cdot)] \\
  \qquad \le  u [ | e^{-\int _{0}^{t}D_{N}( \zeta_S(\tau)(1), \cdot)d\tau} - e^{-\int _{0}^{t}D_{N}( \beta_S(\tau)(1), \cdot)d\tau}) &g (\cdot)|]
 \le  \Bigl (\|{D}_{N}(\cdot,\cdot)\|_{BL} \\
\int_{s}^{t}\| \zeta(\tau)- \beta(\tau)\|^*_{BL}d\tau \Bigr )
u[|g(\cdot)|] \le   \Bigl (\|{D}_{N}(\cdot,\cdot)\|_{BL} & \int_{s}^{t}\| \zeta(\tau)- \beta(\tau)\|^*_{BL}d\tau \Bigr ) \|u\|_{BL}^* \\
 \mbox{  The last two estimate use } Lemma ~\ref{E}.\end{split}\end{equation}

Now if $ \zeta, \beta \in W_{\widehat{N}} , $

\begin{equation} \begin{split}
  (T\zeta -T\beta)(t) & = \Bigl [\int_0^t e^{-D(t-s)}\Bigl (  B_N( \zeta_S(s), \cdot)\bullet [ \zeta_{\mu} - \beta_{\mu}](s) +  [ B_N( \zeta_S(s), \cdot) - B_N( \beta_S(s)), \cdot]\bullet \beta_{\mu}\Bigr )ds,  \\
  &  ( e^{-\int _{0}^{t}D_{N}( \zeta_S(\tau)(1), \cdot)d\tau} - e^{-\int _{0}^{t}D_{N}( \beta_S( \tau)(1), \cdot)d\tau})\bullet u  \\
  & \quad +  \int_{0}^t e^{-\int _{s}^{t}D_{N}( \zeta_S(\tau)(1), \cdot)d\tau} \bullet \Bigl ( F_{N21}[\zeta(s), \gamma]- F_{N21}[ \beta(s),\gamma] \Bigr )ds \\
  & \quad + \int_0^t \Bigl( e^{-\int _{s}^{t}
D_{N}( \zeta_S(\tau)(1), \cdot)d\tau} - e^{-\int _{s}^{t} D_{N}( \beta_S(\tau)(1), \cdot)d\tau} \Bigr )\bullet F_{N21}[\beta(s), \gamma] ds \Bigr ]
  \end{split} \end{equation}
and
  \begin{equation} \begin{split}
  \|T\zeta -T\beta\|_{\R \times{BL}^*} & \le \|B_{N}\|_{\infty}\int_0 ^t  \|\zeta_{\mu}(s) - \beta_{\mu}(s)\|_{BL}^* + \|\beta_{ \mu}\|_{BL}^*  \| B_N(\cdot, \cdot)\|_{Lip} \int_0^t|\zeta_S(s)- \beta_S(s)|ds \\
  & + \|D_N(\cdot, \cdot) \|_{BL}\|u\|_{BL}^*  \int_0^t \|\zeta_{\mu}(s) -\beta_{\mu}(s)\|_{BL}^*ds  \\
  & \quad +  \|F_{N21}(\cdot, \cdot)\|_{Lip} (\|D_N(\cdot,\cdot)\|_{BL} T + 1) \int_0^t \|\zeta_{\mu}(s) -\beta_{\mu}(s)\|_{BL}^*ds \\
  & \quad + T \|F_{N21}\|_{\infty}\|D_N(\cdot, \cdot) \|_{BL}  \int_0^t \|\zeta_{\mu}(\tau) -\beta_{\mu}(\tau)\|_{BL}^*ds .
  \end{split} \end{equation}

  If  \begin{equation} \begin{split} N_T & =  \max \Bigl \{ \|D_N(\cdot, \cdot) \|_{BL}\|u\|_{BL}^*  + \|F_{N21}(\cdot, \cdot)\|_{Lip}( \|D_N(\cdot,\cdot)\|_{BL} T + 1) \\
   & +  T \|F_{N21}\|_{\infty}\|D_N(\cdot, \cdot) \|_{BL} + \|B_{N}\|_{\infty},\|\beta_{ \mu}\|_{BL}^*  \| B_N(\cdot, \cdot)\|_{Lip} \Bigl \} \end{split} \end{equation}
\begin{equation}
  e^{-\lambda t} \|(T\zeta)(t) -(T\beta)(t)\|_{BL}^*  \le  N_T \int_0^t e^{-\lambda(t-s)} e^{-\lambda s}\|\zeta(s) -\beta(s)\|_{BL}^*ds.
   \end{equation}
 Hence,

   \begin{equation}\begin{split}
  \|T\zeta -T\beta\|_{\lambda } & \le  N_T \Bigl ( \sup_{t\in [0,T]} \int_0^t e^{-\lambda(t-s)}ds\Bigr ) \|\zeta -\beta\|_{\lambda}  \\
  & \le \frac{N_T}{\lambda} \|\zeta -\beta\|_{\lambda}.
  \end{split} \end{equation}
  Which is a contraction for   $\lambda$  large  enough.

We label this fixed point $ \varphi \equiv \varphi_{NM+} $  and make the claim that  $ \varphi([0,T]) \subset \mathbb{R}_+ \times BL^*_{+}.$ \\

 Indeed, it is obvious from \eqref{irep} that $ \varphi_{\mu} = (\varphi_{NM+})_{\mu} $ is nonnegative. For $ \varphi_{S} $ notice if we differentiate we get $$ {\varphi_S'}(t) =  \Lambda -D\varphi_S(t) -  B_N(\varphi_S(t), \cdot)\bullet \varphi_{\mu}.$$

    Suppose that $ \varphi_S(t)$ is not positive for all $ t \in [0,T]$. Since $\varphi_S(0)>0$, there is a point $T_0$ with $\varphi_S(T_0) = 0 $ and $\varphi_S(t)>0 $ for $ 0 \le t <T_0.$ For $ 0 \le t \le T_0 $,
     $$ \varphi_{S}'(t) > -\varphi_S(t)D -  \|B_{N}(\cdot,\cdot)\|_{Lip}\varphi_S(t)\widehat{N}= -[D+ \|B_{N}(\cdot,\cdot)\|_{Lip}\widehat{N}] \varphi_S(t). $$ Hence, $ \varphi_S(T_0) > e^{ -(D \|B_{N}(\cdot,\cdot)\|_{Lip}\widehat{N})t} >0 $ by (A1). Which is a contradiction. Hence $\varphi_{NM+} $ is positive invariant.

\emph{Local solution for \eqref{eq:M2} :}\\

 Indeed, it is a simple exercise to show that
 \begin{equation} \label{verify} \begin{split}
  \varphi'_{NM+}(t) & =F_N[\varphi_{NM+}(t), \gamma]
   \end{split} \end{equation}

and obviously from the integral representation \eqref{irep},  $$\varphi_{NM+} (0;s,u, \gamma) = [s,u] ,  ~\forall [s,u ]\in [0,N]\times B_{N}[\textbf{0}]_+.$$

By uniqueness of solution  $$  \varphi_{NM}(t,s,u, \gamma) =\varphi_{NM+}(t,s, u, \gamma) \mbox{ on }  [0,T] \times [0,N] \times B_N[\textbf{0}]_+ \times L\cP^*[M] .$$\\

\emph{ Lipschitz on $ [0,N]\times [-N,N] \times {B_N[\textbf{0}]}_+ \times L \cP^*[M]$ :}\\ 

Looking at the right hand side in \eqref{eq:M2}, since  $ F_N$ is continuous by Lemma \ref{LF} we see that $\varphi_{NM}$ is actually $C^{1}([0,T])$ .       Hence, $ \forall (s,u, \gamma) \in [0,N]\times {B_N[\textbf{0}]}_+ \times L\cP^*[M],$  $ \varphi_{NM}( \cdot,s,u, \gamma)$ is Lipschitz on $[0,T]$ and  the Lipschitz bound does not depend on the variables $s, u, \gamma$. It only depends on $T$ and $ \|F_N\| _{\infty}.$

Fix  $ (s_1,u_1, \gamma_1), (s_2, u_2, \gamma_2) \in [-N,N] \times B_{N}[\textbf{0}]_+ \times L\cP^*[M]$, then  $ \varphi_{NM} (\cdot,s_i,  u_i, \gamma_i) \in C([0,T];\mathbb{R}\times BL^*_+)\mbox{   for  i =1,2 }. $

If $ w_i(\cdot) =\varphi_{NM}(\cdot;s_i,u_i, \gamma_i)$ for i = 1,2, then

\begin{equation*} \begin{split}
 w_i(t) = [s_i,u_i] + \int_{0}^{t} F_N[ \varphi[ w_i(s), \gamma_i]ds \mbox { for i =1,2 .}
 \end{split} \end{equation*}

Hence
\begin{equation*} \begin{split}
 \|w_1(t) -w_2(t) \|_{ \R \times {BL}^*} &  \le  | s_1- s_2| +  \|u_1 -u_2 \|_{BL}^* + \int_{0}^{t}\| F_N[  w_1(s), \gamma_1] - F_N[ w_2(s), \gamma_2] \|_{\R\times BL}^*ds \\
 & \le  | s_1- s_2| + \|u_1-u_2\|_{BL}^* + \|F_N[\cdot, \cdot]  \|_{Lip} \int_{0}^{t} \Bigl (\|w_1(s) -w_2(s)\|_{\R \times BL}^* + \| \gamma_1 - \gamma_2\|_{\infty}^* \Bigr) ds
 \end{split} \end{equation*}

and if  $ \lambda > 0 $
 \begin{equation*} \begin{split}
 e^{-\lambda t}\|w_1(t) -w_2(t) \|_{\R \times BL}^* &  \le  e^{-\lambda t}\|u_1-u_2\|_{BL}^* + \|F_N[\cdot, \cdot] \|_{Lip} \int_{0}^{t} e^{-\lambda( t-s)} e^{-\lambda s}\|w_1(s) -w_2(s)\|_{\R \times BL}^* ds \\
 & \qquad + e^{-\lambda t}|s_1-s_2| + \|F_N[\cdot, \cdot] \|_{Lip}Te^{-\lambda t} \| \gamma_1 - \gamma_2\|^*_{\infty}.
 \end{split} \end{equation*}

Hence,
 \begin{equation*} \begin{split}
  \|w_1 -w_2 \|_{\lambda} &  \le  \|u_1-u_2\|_{BL}^* + \|F_N[\cdot, \cdot]  \|_{Lip} \sup_{t \in [0,T]} \Bigl( \int_{0}^{t} e^{-\lambda( t-s)} ds\Bigr)  \|w_1 -w_2\|_{\lambda}  \\
 & \qquad +  | s_1- s_2| +  \|F_N[\cdot, \cdot] \|_{Lip}T \| \gamma_1 - \gamma_2\|^*_{\infty}
 \end{split} \end{equation*} and
 \begin{equation*} \begin{split}
 \|w_1 -w_2 \|_{\lambda} &  \le   \frac{ \|F_N[\cdot, \cdot] \|_{Lip} }{ \lambda} \|w_1 -w_2 \|_{\lambda}  + |s_1 - s_2 | +  \|u_1-u_2\|_{BL}^* +\|F_N[\cdot, \cdot] \|_{Lip}  T \| \gamma_1 - \gamma_2\|^*_{\infty}.
 \end{split} \end{equation*}

 If  $\lambda$  is such that  $ \frac{ \|F_N[\cdot, \cdot] \|_{Lip} }{ \lambda} < 1$ then we have

 \begin{equation*} \begin{split}
 \|w_1 -w_2 \|_{\lambda} &  \le   \frac{1}{ (1- \frac{ \|F_N[\cdot, \cdot] \|_{Lip} }{ \lambda})}  (| s_1 - s_2 | + \|u_1-u_2\|_{BL}^* + \|F_N[\cdot, \cdot] \|_{Lip} T \| \gamma_1 - \gamma_2\|^*_{\infty}).
 \end{split} \end{equation*}

 Hence,
 \begin{equation*} \begin{split} \|\varphi(t,s_1,u_1, \gamma_) -\varphi(t,s_2, u_2, \gamma_2)\|_{BL}^* \le  \frac{e^{\lambda T}}{ (1- \frac{ \|F_N[\cdot, \cdot] \|_{Lip} }{ \lambda})}  ( |s_1 -s_2 | + \|u_1-u_2\|_{BL}^* + \|F_N[\cdot, \cdot] \|_{Lip} T \| \gamma_1 - \gamma_2\|^*_{\infty}).
 \end{split} \end{equation*}

 Since $ \varphi_{NM} $ is Lipschitz separately in both $ t $ and  $(s, u, \gamma) $, it is Lipschitz.\\

\end{proof}

\subsubsection{Proof of Theorem \ref{main}}
\begin{proof} \begin{enumerate} \item If $  T, M >0 $, $ N \in \mathbb{N}$  by Proposition \ref{local}  there exists continuous

 $$\varphi _{NM} : [0,T] \times B_N[\textbf{0}] \times BL\cP^*[M] \rightarrow BL^* .$$

 Since $$ \R_+ \times BL^* \times BL(Q;\mathcal{P}^*) =\bigcup_{ N \in \mathbb{N}} [0,N] \times B_N[0] \times BL\cP^*[N] $$ if we define

 \begin{equation} \label{eq:union} \varphi = \cup \varphi_{NN}. \end{equation} then we have our continuous

 $$ \varphi : \R_+ \times  {BL^*}  \times BL(Q;\mathcal{P}^*) \rightarrow BL^* .$$ Furthermore, if  $ X = \mathbb{R}\times BL^* \times L(Q;\mathcal{P}^*)$ and  $ D_X[(s_1,u_1, \gamma_1), ( s_2,u_2, \gamma_2)] = | s_1 + s_2| + \|u_1 - u_2 \|_{BL}^* + \| \gamma_1 - \gamma_2 \|_{\infty}^* ,$
 then  $ [X, D_X]$ is a metric space. Define  $$\Phi: \R_+ \times X \rightarrow X $$  by $$ \Phi(t;(s,u, \gamma)) =[ \varphi(t,s,u, \gamma), \gamma] .$$

    \item This also follows from Proposition \ref{local}. Indeed, for fixed $ s,u, \gamma $ there exists $ \hat{N}$ such that  $ (s,u,  \gamma) \in[-N,N] \times B_{\hat{N}}[0] \times L\cP^*[\hat{N}].$ Since differentiability is a local condition we only need to verify \eqref{xx}  on a finite time interval $[0,N],$   $N \ge \hat N .$ This verification is easily done if we can verify that $\varphi$ is bounded on any such time interval.

                     Indeed suppose that  $ \varphi$ is bounded  on any such time interval. Let
                     $$  N(t) = \| \varphi(t)\|_{\R BL}^* $$  Then if  $ M > \max \{\sup_{ t \in [0,N]} N(t), N \} $ then on $ [0,N]\times [-N,N] \times B_N[0] \times L\cP^*[N] $
                      $$\varphi \equiv \varphi_{NN} \equiv \varphi_{MM}. $$

                      Hence  \begin{equation} \begin{split}
                      \varphi'(t,s,u, \gamma)  & = \varphi'_{MM}(t,s,u, \gamma) = F_{M}( \varphi_{MM}(t,s,u,\gamma), \gamma)   = F_{M} ( \varphi_{NN}(t,s,u,\gamma), \gamma ) \\
                        & = F ( \varphi_{NN}(t,s,u,\gamma), \gamma )= F ( \varphi(t,s,u,\gamma), \gamma).
                      \end{split} \end{equation}

                     Also obviously  $\varphi(0,s,u, \gamma) = [s,u] .$
                     Moreover, $\varphi $ is obviously bounded on any finite interval since it is actually continuous on any finite interval.

                     The argument for the following is found in the section leading up to \eqref{CONSTRAINT}.

 \begin {equation}\label{eq:convenient} \left\{\begin{array}{ll}
 \displaystyle {\Phi'}(t;x) & =  \cF [ \Phi(t;x)]    \\
\Phi(0;x)=x.
\end{array}\right.\end{equation}
So we see that  $ \Phi$  satisfies the constraint equations \eqref{CONSTRAINT}.

\item  \begin{equation*} \begin{split}
 \Phi(\R_+\times X_+)  & =  \Phi(\R_+\times \R_+ \times BL^*_{+} \times L(Q;\cP^*))  = \bigcup_{N} \Phi( [0,N]\times[0,N] \times B_{N}[\textbf{0}]_+ \times L\cP^*[N]) \\
& = \bigcup_{N}  \varphi\Bigl([0,N]\times [0,N] \times B_{N+}[\textbf{0}] \times L\cP^*[N]\Bigr)\times L\cP^*[N] \subset \bigcup_{N}\R_+ \times BL^*_{+}\times L\cP^*[N]\\
&  = X_+
\end{split} \end{equation*}
 \item This is an immediate corollary of Proposition \ref{local} given the definition of $ D_X$ and the fact that $ \varphi$ is locally Lipschitz by Propositon \ref{local}.

 \end{enumerate}

Finally we show that $ \Phi$  is actually a semiflow on $X.$
For the first condition notice that for each $ \gamma \in L(Q;\mathcal{P}^*) $,  $\varphi(\cdot, \cdot,\cdot,\cdot, \gamma)$  is a semiflow \cite[ Chpt.1, pg.19]{Thi03}.

Suppose $ x =(s,u, \gamma) \in X $, then
  \begin{equation} \begin{split}
\Phi(t+r, x) &  =[ \varphi(t+r,s,u,\gamma), \gamma] = [\varphi(t, \varphi(r,s,u,\gamma), \gamma), \gamma]= \Phi(t, (\varphi(r,s,u,\gamma), \gamma)) \\
 & = \Phi(t, \Phi(r,x))\end{split} \end{equation}
The second condition is shown to be satisfied by \eqref{eq:convenient}  above.
\end{proof}

\section{Unification}

Here we demonstrate the unifying power of this method. In \cite{CLEVACK} it is demonstrated how to obtain the  discrete, absolutely continuous, selection mutation and pure selection from a measure theoretic model by a proper choice of initial condition and mutational kernel. Here we demonstrate how to obtain a measure theoretic model and hence we obtain all of the above.

\emph{Measure Valued Constraint Equation:}

Clearly
\begin{equation*} \left\{   \begin{split}
S'(t) &=    \Lambda -DS -  B(S(t),q)\bullet \mu(t)[1]\\
\mu'(t) & =B(S(t),q)\gamma(q)\bullet\mu(t)- D(S(t),q)\bullet \mu(t)(dq)\\
[S(0),\mu(0)] & =[ s_0,\mu_0]
\end{split} \right.
\end{equation*}
becomes

\begin{equation}\begin{split}
S'(t) &=    \Lambda -DS - \int_Q B(S(t),q)\mu(t)(dq)\\
\mu'(t)(E) & =\int_Q B(S(t),q)\gamma(q)(E)\mu(t)(dq)-\int_E D(S(t),q)\mu(t)(dq).
\end{split}
\end{equation}

which is the measure valued constraint equation \cite{CLEVACK, JC3}.

 We mention one more important observation. In \cite{CLEVACK, JC1} we notice that the parameter space is  $ C(Q, \mathcal{P}_w)$, but now the parameter space is  $L\cP^*$. In order to model both pure selection and selection mutation in a continuous manner we need for the kernel  $ q \mapsto \delta_q$ to be in $L\cP^*[M]$ for some $M$. This is indeed the case as  \cite[Lemma 3.5]{HILLE},  demonstrates.

\section{ Uniform Eventual Boundedness}
A system $ \frac{dx}{dt} =F(x)$ is called dissipative and its solution uniformly eventually bounded, if all solutions exist for all forward times and if there exists some $c>0$ such that $$ \limsup_{t \rightarrow \infty}||x(t)||< c $$ for all solutions $x$.

\begin{theorem} Under (A1), (A2) the solutions to \eqref{xx} are uniformly eventually bounded on $ X_+ = \R_+\times \R_+ \times BL^*_+ \times L\cP^*.$
 \end{theorem}

\begin{proof} By Theorem \ref{main} $ X_+$ is forward invariant. Hence
if $M(t) = \varphi_S(t) + \varphi_{\mu} (t)[1]$ and  $ \varphi = ( \varphi_S, \varphi_ \mu)$ is the global solution to \eqref{xx}, then
 $ M(t)= \|\varphi(t) \|_{\R \times BL^*},$
 $ M'(t)= \varphi_S'(t) + \varphi_{\mu}'(t) $ and $$ M' = \Lambda -DS - D(\cdot,\varphi_S)\bullet\mu . $$ Hence $$ M'(t) \le \Lambda -dM(t)= -d(M- \frac{\Lambda}{d}) $$ where $ d =\min \{ D, 1, \varpi\}. $ Hence
$M(t) \le \max \{ M(0), \frac{\Lambda}{d}\}$  and  $ \limsup_{t \rightarrow \infty} M(t) \le \frac{\Lambda}{d}.$
\end{proof}

\section{ Conclusion} In this paper we formed a heterogeneous parameter $ \R \times BL^* $ valued basic consumer resource model. One can think of a chemostat, epidemics or any indirect competition of consumers for a resource. The model has as base the ones described in  \cite{SmithThieme, JC1}. However, here we have constructed a $ \R \times BL^* $ valued model, with nonlinear mutation term and substrate dependence in the washout rates. We have showed that this model is well posed, positive invariant and point dissipative.

In this theory we model an evolutionary game as a semiflow on the metric space  $ X = \R \times BL^* \times L(Q;\cal{P}^*) $ of which $ X_+ = BL^*_+ \times L(Q;\cal{P}^*) $ is forward invariant. This model includes all of the well posedness results found in \cite{CLEVACK}.\begin{footnote} {See the list in the second to last paragraph in section 1 above.} \end{footnote} We note that on any forward invariant subspace we have a well-posed model. This includes both  $ \R_+ \times \mathcal{M}_+ \times L\cP^* \mbox { and }  \R_+ \times \overline{\mathcal{M}}_+ \times L\cP^*.$ We conclude that by considering the Lipschitz maps on a compact metric space and forming their dual a nice unifying theory of evolutionary games can be constructed. This elegant theory involves constructing an action $ \bullet$ that allows us to multiply a linear functional by a family of linear functionals. It is difficult to multiply two linear functionals, but it is easy to multiply a linear functional by a family of linear functionals. Moreover this multiplication behaves nicely with respect to norms, i.e. the normed product is less than or equal to the product of the norms. One should notice the length and number of estimates in this paper as compared to those in \cite{CLEVACK,JC3,JC1}.

As far as future development of the theory there are two main paths to be considered. They are asymptotic analysis and parameter estimation. This paper laid the groundwork of the well posedness of this model. This is to be followed by the determination of the asymptotic limit for a pure selection kernel and a solution to the inverse problem.    \cite{AU} reveals how parameter estimation can be performed on structured population models formed on metric spaces metrized with the weak star topology.
So I hope to use the formalism found in \cite{BK} and the techniques found in \cite{AU} to develop a parameter estimation theory for these  $\R \times BL^*$ valued models. Formerly the formalism found in \cite{BK} was untenable due to the fact that the model was formed using the total variation norm, which was different from the norm of continuity of the parameter (mutation kernel). However, now this is no longer an obstacle.

\section*{References}
\bibliography{mylib}

\end{document}